\documentclass{amsart}
\usepackage{amssymb}
\newtheorem{thm}{Theorem}[section]

\newtheorem{lem}{Lemma}[section]
\newtheorem{prop}{Proposition}[section]
\numberwithin{equation}{section}

\theoremstyle{definition}

\theoremstyle{remark}
\newtheorem{rem}{Remark}[section]

\newcommand\la{\langle}
\newcommand\ra{\rangle}

\newcommand{\eps}{\varepsilon}

\newcommand{\D}{\mathbb{D}}
\newcommand{\li}{\mathcal L}
\newcommand\ho{H(\overline{\mathbb D\mathstrut})}
\newcommand\ce{\mathcal C}

\newcommand\bk{\mathfrak B}
\allowdisplaybreaks

\newcommand\lo{\overline{\mathcal L\mathstrut}}
\begin{document}

\title[Logarithmic Bloch space and its predual]{Logarithmic Bloch space and its predual}%
\author{Miroslav Pavlovi\'c}%
\address{Faculty of Mathematics, University of Belgrade, Studentski trg 16, 11001 Beograd, p.p. 550, Serbia}%
\email{pavlovic@matf.bg.ac.rs}%

\thanks{Supported by MNTR Serbia, Project  ON174017}%
\begin{abstract}
We consider the space $\bk^1_{\log^\alpha}$, of analytic functions on the unit disk $\D,$ defined by the requirement
$\int_\D|f'(z)|\phi(|z|)\,dA(z)<\infty,$ where $\phi(r)=\log^\alpha(1/(1-r))$ and show that it is a predual of the ``$\log^\alpha$-Bloch" space and the dual of the corresponding little Bloch space. We prove that a function $f(z)=\sum_{n=0}^\infty a_nz^n$ with $a_n\downarrow 0$ is in $\bk^1_{\log^\alpha}$ iff $\sum_{n=0}^\infty \log^\alpha(n+2)/(n+1)<\infty$
and apply this to obtain a criterion for membership of the Libera transform of a function with positive coefficients
in $\bk^1_{\log^\alpha}.$ Some properties of the Ces\`aro and the Libera operator are considered as well.
\end{abstract}
\maketitle
\section{Introduction and some results}
Let $H(\D)$ denote the space of all functions analytic in the unit disk $\D$ of the complex plane.  Endowed with
the topology of uniform convergence on compact subsets of $\mathbb
D,$ the class $H(\mathbb D)$ becomes a complete locally convex
space. In this paper we are concerned with the predual of the space $\bk_{\log^\alpha}$, $\alpha\in \mathbb R,$
\begin{equation}\label{} \bk_{\log^\alpha}=\Big\{f\in H(\D):|f'(z)|=\mathcal O\Big((1-|z|)^{-1}\log^\alpha\frac 2{1-|z|}\Big)\Big\}. \end{equation}
The norm in $\bk_{\log^\alpha}$ is defined by
 \[\|f\|_{\bk_{\log^\alpha}}=|f(0)|+\sup_{z\in \D}(1-|z|)\log^{-\alpha}\frac1{1-|z|}.\]
 The subspace, $\mathfrak b_{\log^\alpha}$, of $\bk_{\log^\alpha}$ is defined by replacing ``$\mathcal O$" with ``$o$".
 It will be proved:
  \par\medskip\noindent {\bf (A)}
The dual of $\mathfrak b_{\log^\alpha}$ is isomorphic to $\bk^1_{\log^\alpha},$
\begin{equation}\label{12}\bk^1_{\log^\alpha}=\Big\{f:\|f\|_{\bk^1_{\log^\alpha}}=|f(0)|+\int_\D|f'(z)|\log^\alpha\frac 2{1-|z|}\,dA(z)<\infty\Big\},\end{equation}
and  the dual of $\bk^1_{\log^\alpha}$ is isomorphic to
$\bk_{\log^\alpha}$, in both cases with respect to the bilinear form
 \begin{align}\label{pairing}\langle f,g \rangle=\lim_{r\uparrow 1}\sum_{n=0}^\infty \hat f(n)\hat g(n)r^{2n}. \end{align}

 (In \eqref{12} $dA$ stands for the normalized Lebesgue measure on $\D.$)
 This extends the well-known result on the Bloch space  and the little
Bloch space~ $\mathfrak b:=\mathfrak b_{\log^0}.$

 These spaces are Banach spaces, and the space $\mathfrak b_{\log^\alpha}$ coincides with the closure in $\bk_{\log^\alpha}$ of the set of all polynomials. The space $\bk_{\log}:=\bk_{\log^1}$ occurs naturally in the study of pointwise multipliers on the usual Bloch space $\bk:=\bk_{\log^0}$ (see \cite{brown}).

 One of interesting properties of $\bk^1_{\log^\alpha}$ is described in the following theorem:
\begin{thm}\label{thm-decreasing}
Let $f(z)=\sum_{n=0}^\infty a_n z^n,$ where $\{a_n\}$ is a nonincreasing sequence, of real numbers, tending to zero.
Let $\alpha\ge -1.$ Then $f$ belongs to $\bk^1_{\log^\alpha}$ if and only if
\begin{equation}\label{eq-decr} S_\alpha(f):=\sum_{n=0}^\infty \frac{a_n\log^{\alpha}(n+2)}{n+1}\ <\infty. \end{equation}
Moreover, there is a constant $C$ independent of $\{a_n\}$ such that $S_\alpha(f)/C\le \|f\|_{\bk^1_{\log^\alpha}}\le CS_\alpha(f).$
\end{thm}
\begin{proof} See Section \ref{section-decreasing}.  \end{proof}
 In the case $\alpha=0,$ this assertion is proved in \cite{decreasing}. We can take $a_n$ to be the coefficients
 of the Libera transform of a function with positive coefficients. Namely, if $g(z)=\sum_{n=0}^\infty \hat g(n) z^n$ and
 \begin{equation}\label{abs} \sum_{n=0}^\infty \frac{|\hat g(n)|}{n+1}\ <\infty, \end{equation}
then the Libera transform $\li g$ of $g$ is well defined as
\begin{align}\label{eq-lib}\li g(z)&=\frac 1{1-z}\int_z^1f(\zeta)\,d\zeta\\&\nonumber=\sum_{n=0}^\infty z^n\sum_{k=n}^\infty \frac{\hat g(k)}{k+1}  \end{align}
(see, e.g., \cite{NP}). If $\hat g\ge0,$ then condition \eqref{abs} is also necessary for the existence of the integral in \eqref{eq-li}: take $z=0$ to conclude that \eqref{eq-li} implies the convergence of the integral
\[\int_0^1g(t)\,dt=\sum_{n=0}^\infty \frac{\hat g(n)}{n+1}.\]
Then, as an application of Theorem \ref{thm-decreasing} we get:
\begin{thm}\label{thm-lib-pos}
Let $\alpha>-1$, let $g\in H(\D)$, and $\hat g\ge 0$. Then $\li g$ is in $\bk^1_{\log^\alpha}$ if and only if
\begin{equation}\label{} K_\alpha(g):=\sum_{n=0}^\infty \frac{\hat g(n)\log^{\alpha+1}(n+2)}{n+1}\ <\infty.   \end{equation}
We have $K_\alpha(g)/C\le \|\li g\|_{\bk^1_{\log^\alpha}}\le CK_\alpha(g).$
\end{thm}
\begin{proof}  See Section \ref{section-decreasing}. \end{proof}

In the general case, the integral in \eqref{eq-lib} need not exists, but it certainly exists if $g\in \ho,$ which means  that $g$ is analytic in a neighborhood of the closed disk. By using  Theorem \ref{thm-decreasing} we shall prove that $\lo$ cannot be extended to a bounded operator from $\bk^1_{\log^\alpha}$ to $H(\D)$, if $\alpha<0.$ In the case $\alpha\ge 0,$ every function $g\in \bk^1_{\log^\alpha}$ satisfies \eqref{abs}, whence $\li$ is well defined, and we will show that $\li$ maps this space into $\bk^1_{\log^{\alpha-1}}$,  when $\alpha>0.$ If $\alpha=0$ we need a sort of ``iterated" logarithmic space.
\subsection*{Ces\`aro operator}
 The dual of $H(\mathbb D)$ is equal to $ H(\overline
{\mathbb D\mathstrut}),$ where ``$g\in  H(\overline {\mathbb
D\mathstrut})$" means that $g$ is holomorphic in a neighborhood of
$\overline{\mathbb D\mathstrut}$  (depending on $g$). The duality
pairing is given
\begin{equation}\label{pair}
\langle f,g\rangle =\sum_{n=0}^\infty \hat f(n)\hat
g(n),\end{equation}  where $f(z)=\sum_{n=0}^{\infty}\hat
f(n)z^n\in H(\mathbb D)$  and $ g(z)=\sum_{n=0}^{\infty}\hat
g(n)z^n\in H(\overline {\mathbb D\mathstrut}),$ and the series is absolutely convergent (see, e.g.,~
\cite{LR}).
The Ces\`aro operator is defined on $H(\D)$ as
\begin{equation}\label{eq-ce}\ce f(z)=\sum_{n=0}^\infty z^n\frac1{n+1}\sum_{k=0}^n a_k,\quad f\in H(\D).  \end{equation}
 It is easy to verify that the adjoint of $\ce:H(\D)\mapsto H(\D)$ is equal to $\lo: \ho\mapsto \ho,$ under the pairing \eqref{pair}, and vice versa (see, e.g., \cite{NP}).

The operators $\mathcal C$ and $\lo$ acting on
$H^p$ spaces were first studied by Siskakis  in 1987. In
\cite{S} he proved that $\ce$ is  bounded
on $H^p$ for $1<p<\infty,$ and that $\lo$ can be extended to a bounded operator on $H^p,$ $1<p<\infty,$ and obtained some results on their
spectra and norms. A few years later he  proved the boundedness of
the Ces\`aro operator on $H^1$ (\cite{S1}), while Miao proved its
boundedness on $H^p$ for $0<p<1$ (\cite{Miao}). A short proof of
the boundedness of $\mathcal C$  on $H^p,$ $0<p<\infty$, as well as a stronger result,  can be
also found in Nowak \cite{Nowak}.  However, $H^{\infty}$ is not mapped
into itself by $\mathcal C$ (see \cite{DA}). If we write \eqref{eq-ce} as
\begin{equation*}\label{} z  \ce f(z)=\int_0^z \frac {f(\zeta)}{1-\zeta}\, d\zeta, \end{equation*}
and hence
\begin{equation*}\label{eq-izvod}(z \ce f(z))'=\frac{f(z)}{1-z}, \end{equation*}
we conclude that
 $\ce$ maps $H^\infty$  into the Bloch space (see \cite{DA}).

On the other hand, by using the inequality
 \begin{equation*}\label{} |f(z)|=\mathcal O\Big(\log\frac 2{1-|z|}\Big),\quad f\in \bk, \end{equation*}
and the analogous inequality for $f\in\mathfrak b$ (replace ``$\mathcal O$" with ``$o$"), we get:
\par\medskip\noindent {\bf (C)}
The operator $\ce$ maps the space $\bk$ into $\bk_{\log}$, and $\mathfrak b$ into $\mathfrak b_{\log}.$
\par\medskip

One of our aims is to generalize this assertion to some other values of $\alpha$ and then use assertion {\bf (A)}  together with the duality between $\ce$ and $\lo$ to obtain an alternative
proof of some  results on the action of $\li$ from $\bk^1_{\alpha+1}$ to $\bk^1_\alpha,$
where
\begin{equation}\label{eq-li} \li f(z)=\int_0^1 f(t+(1-t)z)\,dt.\end{equation}
  In particular we have:
  \par\medskip\noindent {\bf (D)}
  The operator $\li$ is well defined on $\bk^1_{\log}$ and maps it into $\bk^1.$
\par\medskip

It should be noted that: (a) $\bk^1\subsetneq H^1;$ (b) $\li$ does not map $\bk^1$ into $H^1$ (see \cite{decreasing});
and
(c) $\li$ maps $\bk$ into BMOA \cite{NP}, which improves an earlier result, namely that $\li$ maps $\mathfrak B $ into $\mathfrak B$ (\cite{DRS, xiao}).

The formula \eqref{eq-li} is obtained from \eqref{eq-lib} by integrating over the straight line joining $z$ and 1. A sufficient (not necessary \cite{novo}) condition for the possibility of such integration is \eqref{abs} $(g=f).$

 In proving some of our results, in particular assertions {\bf (A)} and {\bf (B)}, we use a sequence of polynomials constructed
in \cite{hplq} (see also \cite{coef} and \cite{mono}) to decompose the space into a sum which resembles a sum of finite-dimensional spaces (see Section~ \ref{sec-decomp}).
\section{Some more results}
Some elementary facts concerning the cases when $\li f$ is well defined are collected in the following theorem,
where
\begin{equation*}\label{} \ell^1_{-1}=\Big\{g\in H(\D): \|g\|_{\ell^1_{-1}}=\sum_{n=0}^\infty \frac {|\hat g(n)|}{n+1}<\infty\Big\}.\ \end{equation*}
 \begin{thm}\label{prop} Let $\alpha\in \mathbb R.$ Then:
 \begin{itemize}
\item[(a)] $\bk_{\log^\alpha}\subset \ell^{1}_{-1}$ for all $\alpha;$
\item[(b)] $\bk^1_{\log^\alpha}\subset \ell^1_{-1}$ if and only if
$\alpha\ge0;$
\item[(c)] if $\alpha< 0,$ then $\lo$ cannot be extended to a continuous operator from $\bk^1_{\log^\alpha}$ to $H(\D).$
\end{itemize}
\end{thm}
\begin{proof} See Section \ref{section-decreasing}.   \end{proof}
\begin{rem}
The inclusions in (a) and (b) are continuous. Assertion (c) says much more than simply that $\bk^1_{\log^\alpha}\not\subset \ell^1_{-1}.$
\end{rem}
In the context of the action of $\ce $ and $\li$ some new spaces occur: the space $\bk_{\mathop{\rm logg}}$
is defined by the requirement
\begin{equation*}\label{}|f'(z)|=\mathcal O\Big(\log\log \frac 4{1-|z|}\Big), \end{equation*}
the space $\mathfrak b_{\mathop{\rm logg}}$ defined by replacing ``$\mathcal O$" with ``$o$", and the space
$\bk^1_{\mathop{\rm logg}}$ defined by
\begin{equation*}\label{} \int_\D|f'(z)|\log\log\frac 4{1-|z|}\,dA(z)\ <\infty. \end{equation*}

Our next result is
\begin{thm}\label{th-2} {\rm (a)}
If $\alpha>-1,$ then $\ce$ maps the space $\bk_{\log^{\alpha}}$, resp. $\mathfrak b_{\log^{\alpha}}$, into $\bk_{\log^{\alpha+1}},$ resp. $\mathfrak b_{\log^{\alpha+1}}$.\\[-1.8ex]

{\rm(b)} $\ce$ maps the space $\bk_{\log^{-1}}$, resp. $\mathfrak b_{\log^{-1}}$, into $\bk_{\mathop{\rm logg}}$, resp.  $\mathfrak b_{\mathop{\rm logg}}$.\\[-1.8ex]
\end{thm}
\begin{proof}  See Section \ref{sect23}. \end{proof}
\begin{rem}
If $f\in \bk_{\log^\alpha}$ and $\alpha<-1,$ then, as it can easily be shown, $f\in A(\D),$ where $A(\D)$ is the disk-algebra, i.e., the subset of $H^\infty$ consisting of those $f$ which have a continuous extension to the closed disk. Moreover, the modulus of continuity of the boundary function $f_*(\zeta),$ $\zeta\in \partial \D,$
satisfies the condition
\[\omega(f_*,t)=\mathcal O\Big(t\Big(\log \frac 2t\Big)^{\alpha+1}\Big),\quad t\downarrow 0.\]
This follows from the inequality
\[\omega(f_*,t)\le C\int_{1-t}^1M_\infty (r,f')\,dr,\]
see \cite[Theorem 2.2]{edinb}. It should be noted that the modulus of continuity of $f_*$ is ``proportional" to that
of $f(z),$ $z\in\D,$ see \cite{ta, ru}.
\end{rem}
Concerning the Libera operator we shall prove, besides Theorem \ref{prop}(c), the following facts.
\begin{thm}\label{th-3}
{\rm (a)} If $\alpha>0,$ then $\li $ is well defined on $\bk^1_{\log^\alpha}$ and maps this space to $\bk^1_{\log^{\alpha-1}}.$\\[-1.8ex]

{\rm (b)}  $\li$ is well defined on $\bk^1_{\mathop{\rm logg}}$ and maps this space into $\bk^{1}_{\log^{-1}}.$ \\[-1.8ex]

{\rm (c)} $\li$ is well defined on $\bk^1$ and maps it into $\bk^1_{\alpha}$ for all $\alpha<-1.$
\end{thm}
\begin{proof} See Section \ref{sect23}.  \end{proof}

\begin{thm}\label{thm-dual}
Let $\alpha\in \mathbb R.$ Then the dual of $\mathfrak b_{\log^\alpha}$, resp. $\bk^1_{\log^\alpha},$ is isomorphic
to $\bk^1_{\log^\alpha}$, resp. $\bk_{\log^\alpha}$ under the pairing \eqref{pairing}. Similarly, the dual of $\mathfrak b_{\mathop{\rm logg}},$ resp. $\bk^1_{\mathop{\rm logg}},$ is isomorphic to $\bk^1_{\mathop{\rm logg}},$ resp.  $\bk_{\mathop{\rm logg}},$ under the same pairing.
\end{thm}
\begin{proof} See Section \ref{sec-dual}.  \end{proof}
\begin{rem}
The phrase ``the dual of $X$ is isomorphic to $Y$ under the pairing \eqref{pairing}" means that  if $f\in X$ and $g\in Y,$ then the limit in \eqref{pairing} exists and the functional $\Phi(f)=\langle f,g\rangle$ is bounded on $X;$ and on the other hand, if $\Phi\in X^*,$ then there exists $g\in Y$ such that $\Phi(f)=\langle f,g\rangle$, and moreover, there exists a constant $C$ independent of $g$ such that
$\|g\|_Y/C\le \|\Phi\|\le C\|g\|_Y$.
\end{rem}

As an application of Theorems \ref{th-2}, \ref{th-3}, and \ref{thm-dual}, one can prove the following fact.
\begin{thm}
Let $\alpha>0.$ Then the adjoint (with respect to \eqref{pairing}) of the operator $\li: \bk^1_{\log^\alpha}\mapsto \bk^1_{\log^{\alpha-1}}$ is equal
to $\ce:\bk_{\log^{\alpha-1}}\mapsto \bk_{\log^{\alpha}}.$ The adjoint of the operator $\ce:\mathfrak b_{\log^{\alpha-1}}\mapsto \mathfrak b_{\log^\alpha}$ is equal to $\li: \bk^1_{\log^\alpha}\mapsto \bk^1_{\log^{\alpha-1}}$. The analogous assertions hold in the case when $\alpha=0.$
\end{thm}
\section{Decompositions}\label{sec-decomp}
In \cite{hplq}, a sequence $\{V_n\}_0^\infty$  was constructed in the following way.

Let $\omega$ be a $C^\infty$-function on $\mathbb R$ such that
\begin{enumerate}
\item $\omega(t)=1$ for $t\le 1,$
\item $\omega(t)=0$ for $t\ge 2,$
\item $\omega$ is decreasing and positive on the interval $(1,2).$
\end{enumerate}

Let $\varphi(t)=\omega(t/2)-\omega(t),$
and let $V_0(z)=1+z,$ and, for $n\ge 1,$
\[V_n(z)=\sum_{k=0}^\infty \varphi(k/2^{n-1})z^k=\sum_{k=2^{n-1}}^{2^{n+1}-1}\varphi(k/2^{n-1})z^k.\]

The polynomials $V_n $ have the following properties:
\begin{align}\label{I}&
g(z)=\sum_{n=0}^\infty V_n*g(z), \ \text{ for $g\in H(\mathbb D)$};\\&\label{II}
\|V_n*g\|_p\le C\|g\|_p,\ \text{ for $g\in H^p,\ p>0$};\\&\label{III}
\|V_n\|_p\asymp 2^{n(1-1/p)},\ \text{for all $p>0$,}
 \end{align}
where $*$ denotes the Hadamard product.
Here $\|h\|_p$ denotes the norm in the $p$-Hardy space $H^p,$
\begin{align*}\|h\|_p&=\sup_{0<r<1}\bigg(\frac 1{2\pi}\int_0^{2\pi}|h(re^{i\theta})|\,d\theta\bigg)^{1/p}   \\&=
\sup_{0<r<1}M_p(r,g).
 \end{align*}

We need additional properties.
\begin{lem}\label{stud}
Let $P(z)=\sum_{k=m}^j a_k z^k,$ $m<j.$ Then
\[r^j\|P\|_p\le M_p(r,P)\le r^m\|P\|_p, \quad 0<r<1.\]
\end{lem}
When applied to the polynomial $P=V_n*g',$ this gives:
\begin{align}\label{IV} & r^{2^{n+1}-1}\|V_n*g'\|_p\le M_p(r,V_n*g')\le r^{2^{n-1}-1}\|V_n*g'\|_p\ \text{for $n\ge 1$.}
\end{align}

Another inequality will be used (see \cite[Exercise 7.3.5]{mono}):
\begin{equation}\label{V}2^{n-1}\|V_n*g\|_p/C\le  \|V_n*g'\|_p \le C2^{n+1}\|V_n*g\|_p\ \text{for $n\ge 1$}, \end{equation}
where $C$ is a constant independent of $n$ and $g$.
\begin{thm}\label{thm-decomp} Let $\alpha\in \mathbb R,$ and $f\in H(\D)$. Then:\\[-1.8ex]
\begin{itemize}
\item[(i)] $f\in \bk_{\log^\alpha} $ if and only if $\sup_{n\ge0}(n+1)^{-\alpha}\|V_n*f\|_\infty<\infty.$\\[-1.8ex]

 \item[(ii)] $f\in \mathfrak b_{\log^\alpha} $ if and only if $\lim_{n\to\infty}(n+1)^{-\alpha}\|V_n*f\|_\infty=0.$\\[-1.8ex]

 \item[(iii)] $f\in \bk^1_{\log^\alpha}$ if and only if $\sum_{n=0}^\infty (n+1)^{\alpha}\|V_n*f\|_1<\infty.$
 \end{itemize}
 Moreover, the inequality $$C^{-1}\|f\|_{\bk_{\log^\alpha}}\le \sup_{n\ge 0}(n+1)^{-\alpha}\|V_n*g\|_\infty\le C^{-1}\|f\|_{\bk_{\log^\alpha}}$$ holds, where $C$ is independent of $f.$ The analogous inequality holds in the case of {\rm (iii)} as well.
\end{thm}

For the proof we need the following reformulation of \cite[Proposition 4.1]{stu}.
\begin{lem}\label{reform}
Let $\varphi $ be a continuous function on the interval $(0,1]$  such that
$\varphi(x)/x^\gamma$ $(0<x<1)$ is nonincreasing, and $\varphi(x)/x^\beta$ $(0<x<1)$ is nondecreasing,
 where $\beta$ and $\gamma$ are positive constants independent of $x.$ \footnote{\em Following Shields and Williams \cite{SW}, we call such a function {\bf normal.}} Let
 $$F_1(r)=(1-r)^{-1/q}\varphi(1-r)\sup_{n\ge 1}\lambda_n r^{2^{n+1}-1},$$ $$F_2(r)=(1-r)^{1/q}\varphi(1-r)\sum_{n=0}^\infty \lambda_n  r^{2^{n-1}-1},$$ where $\lambda_n\ge0,$ $0<q\le\infty.$ If $F=F_1$ or $F=F_2,$ then
 \[C^{-1}\|F\|_{L^q(0,1)}\le \|\{\varphi(2^{-n})\lambda_n\}\|_{\ell^q}\le C\|F\|_{L^q(0,1)}.\]
\end{lem}
\medskip
\begin{proof}[Proof of Theorem \ref{thm-decomp}]
 {\em Case {\rm (i)}.}
 Let $\varphi(x)=x\log^{-\alpha}(2/x),$ and $q=\infty.$ That $\varphi$ is normal follows from
from the relation
 $$\lim_{x\downarrow 0}\frac{x\varphi_\alpha'(x)}{\varphi_\alpha(x)}=1.$$
 Let $\lambda_n=2^n\|V_n*f\|_\infty.$
 By \eqref{I}, \eqref{II}, \eqref{IV}, and \eqref{V}, we have
 \[C^{-1}|\hat f(1)|+C^{-1}\sup_{n\ge 1}\lambda_nr^{2^{n+1}-1}\le M_\infty(r,f')\le C|\hat f(1)|+C\sum_{n=1}^\infty \lambda_n r^{2^{n-1}-1}.\]
 Hence, by Lemma \ref{reform}, we obtain the desired result.

 \par\smallskip
 {\em Case {\rm (ii)}.} In this case we can proceed in two ways:

 $1^\circ$ Modify the proof of Lemma \ref{reform} to get
 the inequalities
 \[ C^{-1}\|F\|_{C_0[0,1]}\le \|\{\varphi(2^{-n})\lambda_n\}\|_{\mathfrak c_0}\le C\|F\|_{C_0[0,1]},\]
 where $C_0[0,1]=\{u\in C[0,1]:u(1)=0\}$ and $\mathfrak c_0$ is the set of the sequences tending to zero.

 $2^\circ$ Consider the spaces $\mathfrak b_{\log^\alpha}\subset \bk_{\log^\alpha}$  and
 $X=\{f:\|V_n*f\|=o((n+1)^\alpha)\},$ which is, by assertion (i) and its proof, a subspace of a space $Y$ isomorphic to  $\bk_{\log^\alpha}.$ It is not hard to show that the polynomials are dense in both $\mathfrak b_{\log^\alpha}$ and $X.$ This proves (ii).
 \par\smallskip {\em Case {\rm (iii)}.} In this case we use the function $\varphi(x)=x\log^\alpha(2/x)$, let $q=1,$ and then proceed as in the proof  of (i). The details are omitted.
 This concludes the proof of the theorem.
  \end{proof}
\begin{rem}
By choosing $\phi(x)=x\log\log(\frac 4x),$ then we can conclude that Theorem~ \ref{thm-decomp} remains true if $\log^\alpha $, resp. $(n+1)^\alpha,$ are replaced with $\log\log$, resp. $\log(n+2).$
\end{rem}
\section{Functions with decreasing coefficients}\label{section-decreasing}

\begin{proof}[Proof of Theorem \ref{thm-decreasing}] Assuming that \eqref{eq-decr} holds, we want  to prove that
\[ \|f\|_{\mathfrak B^1_{\log^\alpha}}\le Ca_0+C\sum_{n=1}^\infty a_{2^{n-1}}(n+1)^{\alpha}.\] According to Theorem \ref{thm-decomp} and its proof, we have
\begin{equation*}C^{-1} \|f\|_{\mathfrak B^1_{\log^\alpha}}\le a_0+\sum_{n=1}^\infty(n+1)^\alpha\|V_n*f\|_1
\le C  \|f\|_{\mathfrak B^1_{\log^\alpha}}.\end{equation*}
Let  $n\ge1,$ $m=2^{n-1},$ and $Q_k=\sum_{j=m}^k \varphi(j/m)e_j.$ Since {$Q_{4m-1}=V_n$}, we have
\begin{align*}V_n*f&=\sum_{k=m}^{4m-1}\varphi(k/m)a_ke_k\\&=\sum_{k=m}^{4m-1}(a_k-a_{k+1})Q_k +a_{4m}Q_{4m-1}\\  & =
 \sum_{k=m}^{4m-1}(a_k-a_{k+1})Q_k +a_{4m}V_n. \end{align*}
 On the other hand,
 $ Q_k=V_n*\Delta_{n,k}$, where
 \[ \Delta_{n,k}=\sum_{j=2^{n-1}}^k z^k,\quad 2^{n-1}\le k\le 2^{n+1}.\]
 By \eqref{II}, with $g=\Delta_{n,k},$ we have
 \[\|Q_k\|_1\le C\|\Delta_{n,k}\|_1\le C\log(k+1-2^{n-1})\le C(n+1).\]
 Combining these inequalities we get
 \begin{align*}\|V_n*f\|_1(n+1)^\alpha&\le C\sum_{k=m}^{4m-1}(a_k-a_{k+1})(n+1)^{\alpha+1} +Ca_{4m}\|V_n\|_1(n+1)^\alpha\\ &
\le  C(n+1)^{\alpha+1}(a_m-a_{4m})+Ca_{4m}(n+1)^\alpha  \\&=
C(n+1)^{\alpha+1}(a_{2^{n-1}}-a_{2^{n+1}})+C(n+1)^\alpha a_{2^{n+1}}. \end{align*}
{\em Here we have used the relation $\|V_n\|_1\le C$} (see \eqref{III})!
Thus
\begin{align*}\label{} (n+1)^\alpha\|V_n*f\|_1&\le C(n+1)^{\alpha+1}(a_{2^{n-1}}-a_{2^n})\\&\quad+C(n+1)^{\alpha+1}(a_{2^{n}}-a_{2^{n+1}})\\&\qquad +C(n+1)^\alpha a_{2^{n+1}}, \end{align*}
and therefore it remains to estimate the sums
\[ S_1=\sum_{n=1}^\infty (n+1)^{\alpha+1}(a_{2^{n-1}}-a_{2^n})\ \ \text{and}\ \ S_2=\sum_{n=1}^\infty (n+1)^{\alpha+1}(a_{2^{n}}-a_{2^{n+1}}).\]

If  $\alpha>-1,$ then
\[(n+1)^{\alpha+1}\le C\sum_{k=1}^n (k+1)^\alpha,\]
and hence
\begin{align*}S_1 & \le C\sum_{n=1}^\infty (a_{2^{n-1}}-a_{2^n})\sum_{k=1}^n(k+1)^\alpha \\&=
C\sum_{k=1}^\infty (k+1)^\alpha \sum_{n=k}^\infty (a_{2^{n-1}}-a_{2^n})\\&
=C\sum_{k=1}^\infty (k+1)^\alpha_{2^{k-1}}. \end{align*}
In the case of $S_2$ we get
\[S_2\le C\sum_{k=1}^\infty (k+1)^\alpha a_{2^k},\]
which completes the proof of ``if" part of the theorem in the case $\alpha>-1. $ If $\alpha=-1,$
then
\[\|V_n*f\|_1(n+1)^{-1}\le C(a_{2^{n-1}}-a_{2^{n+1}}) + C(n+1)^{-1}a_{2^{n+1}}, \]
from which we get the desired result in the case $\alpha=-1.$

To prove ``only if" part we use Hardy's inequality in the form
\[\pi M_1(r,g)\ge \sum_{n=0}^\infty\frac {|\hat g(n)|}{n+1}r^n.\]
It follows that
\begin{align*} \int_\D|f'(z)|&\log^\alpha\frac2{1-|z|}\,dA(z)   \\&=2\int_0^1M_1(r,f')\log^\alpha\frac2{1-r}r\,dr  \\&
\ge \frac 2\pi \sum_{n=1}^\infty a_n\frac n{n+1}\int_0^1\log^\alpha\frac2{1-r}r^n\,dr.\end{align*}
Now the desired result follows from the inequality
\begin{equation*}\label{}\int_0^1\frac{\varphi(1-r)}{1-r}r^n\,dr\ge c\varphi\Big(\frac 1{n+1}\Big)\quad(c={\rm const.>0}), \end{equation*}
valid for any function normal function $\varphi$ (see \cite[Lemma 4.1]{stu}).
   \end{proof}

Before proving Theorem \ref{thm-lib-pos}, some remarks are in order.
Let $f(z)=\sum_{n=0}^\infty a_nz^n,$ $a_n\ge 0.$ In order that $\li f$ be well defined by \eqref{eq-li} {\em it is necessary and sufficient} that
\begin{equation}\label{eq-li-pos}\sum_{n=0}^\infty \frac{a_n}{n+1}\ <\infty. \end{equation}
We already mentioned in Introduction that this condition implies the existence of the integral in \eqref{eq-li}. In fact, this integral converges uniformly on compact subsets of $\D,$ which means that the limit
\[\lim_{x\uparrow 1}\int_0^x f(t+(1-t)z)\,dt\]
exists and is uniform in $|z|<\rho,$ for every $\rho<1.$ This guarantees that $\li f$ is analytic. On the other hand, if
the integral in \eqref{eq-li} exists, then we take $z=0$ to  conclude that \eqref{eq-li-pos} holds.

\begin{proof}[Proof of Theorem \ref{thm-lib-pos}]
The Taylor coefficients of $\li f$ are
\[b_n=\sum_{k=n}^\infty \frac{a_n}{n+1}.\]
The sequence $\{b_n\}$ is nonincreasing so we can apply Theorem \ref{thm-decreasing} to conclude that
$\li f\in \bk^1_{\log^\alpha}$ if and only if
\begin{align*} \sum_{n=0}^\infty\frac{\log^\alpha(n+2)}{n+1}&\sum_{k=n}^\infty\frac{a_k}{k+1} \\&=
\sum_{k=0}^\infty \frac{a_k}{k+1}\sum_{n=0}^k\frac{\log^\alpha(n+2)}{n+1} <\infty. \end{align*}
Now the desired result follows from the estimate
\[C^{-1}\log^{\alpha+1}(k+2)\le  \sum_{n=0}^k\frac{\log^\alpha(n+2)}{n+1}\le C\log^{\alpha+1}(k+2), \]
which holds because $\alpha>-1.$
  \end{proof}
\begin{rem}
The above proof shows that $\li f$ belongs to $\bk^1_{\log^{-1}} $ if and only if
\[\sum_{n=0}^\infty\frac {a_n}{\log\log(n+4)}<\infty.\]
\end{rem}
  Now we pass to the proof of Theorem \ref{prop}.

  \begin{proof}[Proof of Theorem {\rm \ref{thm-decreasing}}{\rm (c)}] Since $\bk_{\log^\alpha}\subset \bk_{\log^\beta}$ for $\beta<\alpha,$ we may assume that $-1<\alpha<0.$
  Let
  \[ f(z)=\sum_{n=0}^\infty a_n{z^n},\quad a_n={\log^{-\eps-\alpha}}(n+2),\]
  where $\eps>1.$
  Condition \eqref{eq-li-pos} holds because $\eps>1.$ For every $r\in (0,1)$ the function $f_r(z)=f(rz)$ belongs to
  $\ho$ and, by Theorem \ref{thm-decreasing} and its proof, the set $\{f_r:0<r<1\}$ is bounded in $\bk^1_{\log^\alpha}.$ On the other hand,
  \begin{align*}\lo (f_r)(0)=\sum_{k=0}^\infty \frac{r^k}{(k+1)\log^{\alpha+\eps}(k+2)}. \end{align*}
  Now choose $\eps=1-\alpha>1$ (because $\alpha<0$) to get
  \begin{align*}\lo(f_r)(0) &=\sum_{k=0}^\infty \frac{r^k}{(k+1)\log(k+2)}  \\&
  \longrightarrow\infty \ \ (r\uparrow 1). \end{align*}
  This contradicts the fact that if a set $X\subset \bk^1_{\log^\alpha}$ is bounded and $\lo$ is bounded on $\bk^1_{\log^\alpha}$, then the set $\{\lo f(0):f\in X\}$ is bounded because the functional  $h\mapsto h(0)$  is continuous  on $H(\D).$         This completes the proof.
    \end{proof}
    \begin{proof}[Proof of Theorem \rm\ref{prop}\rm{(a)}]  Let $g\in \bk_{\log^\alpha}.$ Then
    \[M_2(r,g')\le C(1-r)^{-1}\log\frac2{1-r}.\]
    It follows that
    \[2^{n}\bigg(\sum_{k=2^{n}}^{2^{n+1}-1}|\hat g(k)|^2\bigg)^{1/2}r^{2^{n+1}}\le C(1-r)^{-1}\log^\alpha\frac2{1-r}.\]
    Taking $r=1-2^{-n},$ $n\ge1, $  we get
    \[\bigg(\sum_{k=2^{n}}^{2^{n+1}-1}|\hat g(k)|^2\bigg)^{1/2}\le C\log^\alpha(n+1).\]
    Hence,
    \begin{align*} 2^{-n}\sum_{k=2^n}^{2^{n+1}-1}|\hat g(k)|&\le \bigg(2^{-n}\sum_{k=2^{n}}^{2^{n+1}-1}|\hat g(k)|^2\bigg)^{1/2}\\&\le 2^{-n/2}\log^\alpha(n+1).
    \end{align*}
    This                    gives the result. \end{proof}
    \begin{proof}[Proof of Theorem \rm\ref{prop}\rm{(b)}]  In this case we use Hardy's inequality as in the proof of Theorem
    \ref{thm-decreasing} to get
    \[\|g\|_{\bk^1_{\log^\alpha}}\ge c\sum_{n=0}^\infty \frac{|\hat g(n)|\log^\alpha(n+2)}{n+1}.\]
    This proves the result because $\alpha\ge 0.$
     \end{proof}

     \section{Proofs of Theorem \ref{th-2} and \ref{th-3}}\label{sect23}
     Define the operator $\mathcal R:H(\D)\mapsto H(\D)$ by
     \[\mathcal Rf(z)=\sum_{n=0}^\infty (n+1)\hat f(n)z^n=\frac d{dz}(zf(z)).\]
     By using Theorem \ref{thm-decomp} and the relation
     \begin{equation}\label{rrr} C^{-1}2^n\|V_n*f\|_p\le \|V_n*\mathcal Rf\|_p\le C2^n\|V_n*f\|_p\quad (n\ge 0)\end{equation}, one proves that
     the norm in $\bk_{\log^\alpha}$ is equivalent to
      $$\sup_{z\in \D}(1-|z|)\log^{-\alpha}\frac{2}{1-|z|}|\mathcal R f(z)|.$$
\begin{proof}[Proof of Theorem \ref{th-2}{\rm (a)}]   Let $\alpha>-1$ and $f\in \bk_{\log^\alpha}.$ Then, by integration,
\begin{equation*}\label{} |f(z)|\le \log^{\alpha+1}\frac 1{1-|z|}. \end{equation*}
Since
\begin{equation*}\label{} \mathcal R\ce f(z)=\frac{f(z)}{1-z}, \end{equation*}
we see that
\[|\mathcal R\ce f(z)|\le C(1-|z|)^{-1}\log^{\alpha+1}\frac1{1-|z|}. \]
The result follows.
 \end{proof}
 \begin{proof}[Proof of Theorem \ref{th-2}{\rm (b)}] The function $\varphi(x)=x\log\log(4/x)$ is normal  because $\lim_{x\to 0}x\varphi'(x)/\varphi(x)=1.$ Hence, arguing as in the proof of Theorem \ref{thm-decomp} we conclude that
 $f\in \bk_{\mathop{\rm logg}}$ if and only if
 \[ \sup_{n\ge0} \|V_n*f\|_\infty/\log(n+2)<\infty.\]
 Then using \eqref{rrr} we find that $g\in \bk_{\mathop{\rm logg}}$ if and only if
 \[ |\mathcal Rg(z)|\le C(1-|z|)^{-1}\log\log\frac 4{1-|z|}.\]
 The rest of the proof is the same as in the case of (a).
     \end{proof}
     \begin{rem}
     In the case of the little spaces the proofs are similar and is therefore omitted.
     \end{rem}

     For the proof of Theorem \ref{th-3} we need the following lemma \cite{novo}:
     \begin{lem}
     If $f\in \ell^1_{-1},$ then $\li f$ is well defined by \eqref{eq-li} and the inequality
     \begin{equation}\label{ineq-li} rM_1(r,(\li f)')\le 2(1-r)^{-1}\int_r^1M_1(s,f')\,ds, \quad 0<r<1,\end{equation}
     holds.
     \end{lem}
     Before passing to the proof observe that $\bk^1_{\log^\alpha}\subset \bk^1$ and $\bk^1_{\mathop {\rm logg}}\subset \bk^1,$ and, since $\bk_1\subset H^1,$ we see that in all cases of Theorem \ref{th-3} the operator $\li$ is well defined.
     \begin{proof}[Proof of Theorem \ref{th-3}\rm (a)]
     We have, by \eqref{ineq-li},
     \begin{align*} \int_\D|&(\li f)'(z)|\log^{\alpha-1}\frac2{1-|z|}\, dA(z)\\ & =2\int_0^1M_1(r,(\li f)')\log^{\alpha-1}\frac{2}{1-r}r\,dr \\&
     \le 4\int_0^1(1-r)^{-1}\log^{\alpha-1}\frac{2}{1-r}\,dr\int_r^1M_1(s,f')\,ds\\&
     =4\int_0^1M_1(s,f')\,ds\int_0^s(1-r)^{-1}\log^{\alpha-1}\frac{2}{1-r}\,dr \\&
     \le C\int_0^1M_1(s,f')\log^{\alpha}\frac{2}{1-s}\,ds.\end{align*}
     A standard application of the maximum modulus principle shows that the inequality remains valid if we replace $ds$
     with $s\,ds.$ This gives the result.
          \end{proof}

          The proofs of Theorem \ref{th-3}, (b) and (c), are similar and we omit them.

\section{Proof of Theorem \ref{thm-dual}}\label{sec-dual}
We consider a more general situation. Let $X\subset H(\D)$ (with continuous inclusion) be a Banach space
such that the functions $f_w(z)=f(wz),$ $|w|\le 1,$ belong to $X$ whenever $f\in X,$ and $\sup_{|w|\le1}\|f_w\|_X\le \|f\|_X.$ Such a space is said to be {\em homogeneous} (see \cite{blasco-revista}). A homogeneous space satisfies the condition
\begin{equation}\label{71}\|V_n*f\|_X\le C\|f\|_X, \quad f\in X, \end{equation}
where $C$ is independent of $n$ and $f.$
 \par\smallskip
 If in addition
 \begin{equation}\label{abel} \lim_{r\uparrow 1}\|f-f_r\|_X=0,\quad f\in X,\end{equation} then the dual  of X can be identified with the space, $X',$ of those $g\in H(\D)$
for which the limit \eqref{pairing} exists for all $f\in X$ (see \cite{aleman, blasco-revista}).
Also, the dual of a homogeneous space $X$ satisfying \eqref{abel}  can be realized as the space of coefficient multipliers, $(X,A(\D)),$
from $X$ to $A(\D); $ in this case we have $(X,A(\D))=(X,H^\infty)=:X^\ast$ (see \cite{blasco-revista}). The norm in $X^\ast$ is introduced as
\begin{equation*}\label{} \|g\|_{X^\ast}=\sup{\|f*g\|_\infty}: f\in X,\ \|f\|_X\le 1\},\end{equation*}
and, if $X$ is homogeneous and satisfies \eqref{abel}, it is equal to
\begin{equation*}\label{}\|g\|_{X'}=\sup\{|\la f,g\ra| : f\in X,\ \|f\|_X\le 1\}. \end{equation*}
There is another way to express $\la f,r\ra,$ when $f\in X,$ $X$ satisfies \eqref{abel}, and $g\in X'$; namely, in this case,the function $f*g$ belongs to $A(\D),$ and we have $\la f, g\ra=(f*g)(1)$ (see \cite{I, blasco-revista}).

We fix a sequence $\lambda=\{\lambda_n\}_{0}^\infty$ of positive real numbers such that
\begin{equation}\label{lambda} 0<\inf_{n\ge 0}\frac{\lambda_{n+1}}{\lambda_n},\quad\sup_{n\ge 0}\frac{\lambda_{n+1}}{\lambda_n}\ <\infty.\end{equation}

It is clear that the spaces $H^p$ $(0<p\le\infty)$, $A(\D),$ $\bk_{\log^\alpha},$ $\mathfrak b_{\log^\alpha}$, and $\bk^1_{\log^\alpha}$ are homogeneous. Among them only $H^\infty$ and $\bk_{\log^\alpha}$ do not satisfy condition \eqref{abel}.
\par\smallskip
Consider the following three spaces of sequences $\{f_n\}_0^\infty,$ $f_n\in H(\D)$:\par\smallskip
(a) $ \mathfrak c_0(\lambda, X)=\{\{f_n\}: \lim_{n\to\infty}\|f_n\|_X/\lambda_n=0 \};$
\par\medskip
(b) $\ell^\infty(\lambda,X)=\{\{f_n\}: \sup_{n\ge0}\lambda_n\|V_n*f\|_X<\infty\};$
 \par\medskip
 (c) $\ell^1(\lambda,X)=\{\{f_n\}: \sum_{n=0}^\infty \|f_n\|_X/\lambda_n<\infty \}.$
\par\medskip
 We also define the spaces $v_0(\lambda,X),$ $V^\infty(\lambda,X),$ and $V^1(\lambda,X)$ (as subsets of $H(\D)$) by replacing $f_n$ with $V_n*f$ in (a), (b), and (c), respectively.
The proof of the following lemma is rather easy, and is therefore left to the reader.
\begin{lem} If $X$ is a homogeneous space, then so are $ v(\lambda, X),$ $V^\infty(\lambda,X),$ and $V^1(\lambda,X).$
The spaces $v_0(\lambda,X)$ and $V^1(\lambda,X)$ satisfy \eqref{abel}.  The space  $v_0(\lambda,X)$ is equal to the closure in $V^\infty(\lambda,X)$ of the sets of all polynomials.
\end{lem}

Theorem \ref{thm-dual} will be deduced from Theorem \ref{thm-decomp} and the following.

\begin{prop}\label{th-dual}
 If $X$ is a homogeneous space satisfying \eqref{abel}, then the dual of $v_0(\lambda,X)$, resp. $V^1(\lambda,X),$ is isomorphic to $V^1(\lambda,X'),$ resp. $V^\infty(\lambda,X'),$ with respect to \eqref{pairing}.
\end{prop}

In proving we use ideas from \cite{I, II, coef}. For the proof we need the following lemma.

\begin{lem}\label{lqs}
The operator $T(\{f_n\})=\sum_{n=0}^\infty V_n*f_n$ acts as a bounded operator from $Y$ to $Z,$ where $Y$ is one of the spaces $\mathfrak c_0(\lambda,X),$ $\ell^\infty (\lambda,X)$, and $\ell^1(\lambda,X),$ while $Z$ is $v_0(\lambda,X),$ $V^\infty(\lambda,X)$, and $V^1(\lambda,X),$
respectively.
\end{lem}

\begin{proof} We have
\begin{equation*}\label{}\begin{aligned} V_n*V_j=0\quad \text{for $|j-n|\ge2$} \end{aligned} \end{equation*}
and hence
\begin{equation*}\label{}\begin{aligned}
V_n*T(\{f_j\})=\sum_{j=n-1}^{n+1}V_n*V_j*f_j, \quad n\ge 0,
 \end{aligned} \end{equation*}
where, by definition, $w_j=f_j=0$ for $j<0.$
It follows that
\begin{align*}
\|V_n*T(\{f_j\}\|_X \le C\sum_{j=n-1}^{n+1}\|f_j\|_X,
\end{align*}
where we  have used \eqref{71}. Now the proof is easily completed by using \eqref{lambda}.
\end{proof}

\begin{lem}\label{S}
Let $g\in (v_0(\lambda,X))',$ resp. $g\in (V^1(\lambda,X))',$ and  define the operator $S$ on $\mathfrak c_0(\lambda,X)$, resp. $\ell^1(\lambda,X)$, by
\begin{equation*}\label{} S(\{f_n\})=T(\{f_n\})*g=\sum_{k=0}^\infty f_k*V_k*g.\end{equation*}
Then $S$ maps $\mathfrak c_0(\lambda,X)$, resp. $\ell^1(\lambda,X)$, into $H^\infty$ and
$\|S\|\le C\|g\|_{(v_0(\lambda,X))'},$ resp. $\|S\|\le C\|g\|_{(V^1(\lambda,X))'}.$
\end{lem}
     \begin{proof}

     By the preceding lemma, we have
     \begin{align*}\|S(\{f_n\})\|_\infty &\le \|T(\{f_n\})\|_{v_0(\lambda,X)}\|g\|_{(v_0(\lambda,X))\ast}  \\&
     \le C\|\{f_n\}\|_{\mathfrak c_0(\lambda,X)}\|g\|_{(v_0(\lambda,X))\ast}. \end{align*}
      This proves the result in one case. In the other case the proof is the same.
       \end{proof}
     \begin{proof}[Proof of Proposition \ref{th-dual}]     Define the polynomials $P_n$ $(n\ge0)$ by
\begin{equation*}\label{}\begin{aligned} P_n=V_{n-1}+V_n+V_{n+1}. \end{aligned} \end{equation*}
Hence
\begin{align*} V_n=\sum_{j=0}^\infty V_j*V_n=(V_{n-1}+V_n+V_{n+1})*V_n=P_n*V_n. \end{align*}
Let $f\in v_0(\lambda,X)$ and $g\in V^1(\lambda,X').$ It is easily verified that, when $0<r<1,$
\begin{align*} (f*g)(z) &=\sum_{n=0}^\infty (f*V_n*g)(z)\\&
=\sum_{n=0}^\infty  (P_n*f*V_n*g)(z), \quad z\in \D. \end{align*}
the series being absolutely convergent. Since
\begin{align*}
\|P_n*f_r*V_n*g\|_\infty&\le
 \|P_n*f\|_X\,\|V_n*g\|_{X^\ast},
\end{align*}
we have
\begin{align*}\|f*g\|_\infty&\le\sum_{n=0}^\infty \|P_n*f\|_X\,\|V_n*g\|_{X^\ast}   \\&\le C
\sum_{n=0}^\infty \|P_n*f\|_X\,\|V_n*g\|_{X^\ast}\\&
=C\sum_{n=0}^\infty \big(\|P_n*f\|_X/\lambda_n\big)\,\big(\lambda_n\|V_n*g\|_{X^\ast}\big)\\&
\le C\|f\|_{v_0(\lambda,X)}\|g\|_{V^1(\lambda,X^\ast)}
\end{align*}
This proves the inclusion $V^{1}(\lambda,X')\subset (v_0(\lambda,X))'.$

To prove the converse, let $g\in (v_0(\lambda,X))^\ast.$ Let $S$ denote the operator defined in Lemma \ref{S}.
By Lemma \ref{S}, $ S$ acts as a bounded operator from $\mathfrak c_0(\lambda,X)$ into $H^\infty$ and we have
$\|S\|\le C\|g\|_{v_0(\lambda,X)^\ast}.$ Now it suffices to prove that
\[\begin{aligned} \|S\|\ge (1/2) \|\{g_n\}\|_{\ell^1(\lambda,X^\ast)}=(1/2)\|g\|_{V^1(\lambda,X)}.  \end{aligned} \]
 For each $n\ge 0$ choose $f_n\in X$ so that $\|f_n\|_X=1$ and $\langle f_n,g_n\rangle$ is a real number such that $\langle f_n,g_n\rangle\ge (1/2)\|g_n\|_{X^\ast}.$
If $\{a_n\}$ is a finite sequence of nonnegative real numbers, then
\begin{align*}S(\{a_nf_n\})&=\sum_{n=0}^\infty a_n\la f_n,g_n\ra   \\&
\ge (1/2)\sum_{n=0}^\infty a_n\|g_n\|_{X^\ast}\\&
=(1/2)\sum_{n=0}^\infty (a_n/\lambda_n)\lambda_n\|g_n\|_{X^\ast}.  \end{align*}
Hence, by taking the supremum over all $\{a_n\}$ such that $0\le a_n\le \lambda_n,$
we get $S\{\lambda_n f_n\}\ge (1/2)\sum_{n=0}^\infty \lambda_n\|g_n\|_{X^\ast}.$ Since $\|\{a_nf_n\}\|_{\mathfrak c_0(\lambda,X)}\le 1,$ where $a_n=\lambda_n$  for $0\le n\le N$ $(N\in\mathbb N)$ and $a_n=0$ for $n>N$ we see that
$\|S\|\ge (1/2)\|g\|_{V^1(\lambda, X^\ast)},$ as desired. This completes the proof that $v_0(\lambda,X)'=V^1(\lambda,X').$ In a similar way one proves that $V^1(\lambda,X)'=V^\infty(\lambda,X')$, which is all what was to be proved.
\end{proof}
\begin{proof}[Proof of Theorem \ref{thm-dual}]
First we  prove that $(\mathfrak b_{\log^\alpha})'=\bk^1_{\log^\alpha}.$  By Theorem \ref{thm-decomp}, we have
$\mathfrak b_{\log^{\alpha}}=v_0(\lambda, A(\D)),$ where $\lambda_n=(n+1)^\alpha.$ Hence, by Proposition \ref{th-dual}, the dual of $\mathfrak b_{\log^{\alpha}}$ is isomorphic to $V^1(\lambda,A(\D)').$ In order to estimate $\|V_n*g\|_{A(\D)'}$ first observe that $H^1\subset A(\D)'$ and moreover $\|V_n*g\|_{A(\D)'}\le \|V_n*g\|_{1}. $
On the other hand, let $\Phi$ be a bounded linear functional on $A(\D)$, let $\Phi_0$ be the Hahn/Banach extension of $\Phi$ to
$hC(\D)$, and choose $g\in A(\D)^{\rm a}$ so that $\Phi(f)=\langle f,g\rangle$ for all $f\in A(\D).$ By the Riesz representation theorem, we have
\begin{align*}
\Phi_0(f)&=\frac1{2\pi}\int_0^{2\pi} f(e^{-i\theta})\,d\mu(e^{i\theta})\\&= \lim_{r\to1^- }\frac 1{2\pi}\int_0^{2\pi} f(re^{-i\theta})g(re^{i\theta})\,d\theta\\&= \lim_{r\to 1^-}\sum_{n=0}^\infty \hat f(n)\hat g(n) r^{2n},\end{align*} and $\|\mu\|=\|\Phi\|=\|\Phi_0\|.$ In particular, taking $f(w)=(1-zw)^{-1},$ where $z\in\D$ is fixed, we get
\begin{align*}
\frac1{2\pi}\int_0^{2\pi} (1-e^{-i\theta} z)^{-1}{d\mu(e^{i\theta})}=g(z).
\end{align*}
Hence
\begin{align*}\mathcal R^1 g(z)=\frac1{2\pi}\int_0^{2\pi} (1-e^{-i\theta} z)^{-2}\,d\mu(e^{i\theta}), \end{align*}
and hence, by integration,
\begin{equation*}\label{}\begin{aligned} M_1(r,\mathcal R^1f)\le \|\mu\|(1-r^2)^{-1}=\|g\|_{A(\D)'}(1-r^2)^{-1}.
\end{aligned} \end{equation*}
 Now we proceed as in the proof of Theorem \ref{thm-decomp} to conclude that $\|V_n*g\|_1\le C\|V_n*g\|_{A(\D)'}.$
 It follows that $g\in (\mathfrak b_{\log^\alpha})'$ if and only if $g\in V^1(\lambda,H^1),$ i.e., by Theorem \ref{thm-decomp}, $g\in \bk^1_{\log^\alpha}.$

 In proving that $(\bk^1_{\log^\alpha})'$ is isomorphic to $\bk_{\log^\alpha}$ we use the inclusions $H^\infty\subset (H^1)'\subset \mathfrak B,$ and then proceed as above.
\end{proof}
\begin{rem}
The above proof of Theorem \ref{thm-dual} certainly is not the simplest one. However, it can be applied to prove some general duality and multipliers theorems (see \cite{I,  II,  coef}. For instance, the dual of $\mathfrak b_{\rm logg}$
is isomorphic to $\bk^1_{\rm logg},$ and the dual of $\bk^1_{\rm logg}$ is isomorphic to $\bk_{\rm logg}.$
\end{rem}

\bibliographystyle{amsplain}

\begin{thebibliography}{10}
\bibitem{aleman} A. Aleman, M. Carlsson, and A.-M. Persson,
\emph{Preduals of $Q_p$-spaces},
            Complex Var. Elliptic Equ.
             {\bf 52} (2007), No. 7, 605--628.

\bibitem{blasco-revista} O. Blasco and M. Pavlovi\'c, {\em Coefficient multipliers on Banach spaces of analytic functions},
Revista Mat. Iberoamericana, to appear.

\bibitem{brown} L. Brown and A.L. Shields, \emph{Multipliers and cyclic vectors in the Bloch space,} Michigan Math. J.
{\bf 38} (1991), 141--146.


\bibitem {DA} N. Danikas and A.G.Siskakis,
 \emph{The Ces\`aro operator on bounded analytic functions}, {Analysis} {\bf 13} (1993), no. 3, 295-299.

\bibitem{DRS} N. Danikas, S. Ruscheweyh, and A. Siskakis, \emph{
Metrical and topological properties of a generalized Libera transform},
Arch. Math. 63, No.6, (1994), 517--524.


 \bibitem{hplq} M. Jevti\'c and M. Pavlovi\'c,  \emph{On multipliers from $H^p$  to $\ell^q$ $ (0<q<p<1)$}
  Arch. Math. {\bf 56} (1991), 174--180.

\bibitem{coef} M. Jevti\'c and M. Pavlovi\'c,
 \emph{Coefficient multipliers on spaces of analytic functions,}
 Acta Sci. Math. (Szeged) {\bf 64} (1998), 531--545.


\bibitem{LR} D.H. Luecking, LA. Rubel, {\it Complex Analysis, A
Functional Analysis Approach}, Springer-Verlag, New York, 1984.


\bibitem{stu} M. Mateljevi\'c and M. Pavlovi\'c, \emph{$L^p$ behaviour of the integral means of analytic functions},
Studia Math. 77(1984), 219--237.

\bibitem{Miao} J. Miao, {\em The Ces\`{a}ro operator is bounded on
$H^p$ for $0<p<1$}, {Proc. Amer. Math. Soc. } {\bf 116} (1992),
1077-1079.
\bibitem{Nowak} M. Nowak, {\em Another proof of boundedness of the
 Ces\`{a}ro operator  on $H^p$,} {Ann. Univ. Mariae Curie-Sk\l
 odowska  Sect. A} {\bf 54} (2000),75--78.

\bibitem{NP} M. Nowak and M. Pavlovi\'c, \emph{On the Libera operator,} J. Math. Anal. Appl. {\bf 370}, No. 2, (2010),  588--599.

\bibitem{I} M. Pavlovi\'c, \emph{Mixed norm spaces of analytic and harmonic functions. I,}
{ Publ. Inst. Math. (Beograd) {\bf 40} (1986), 117--141.}
\bibitem{II}
M. Pavlovi\'c, \emph{Mixed norm spaces of analytic and harmonic functions. II,}
Publ. Inst. Math. (Beograd) {\bf 41} (1987), 97--110.
\bibitem{edinb} M. Pavlovi\'c, \emph{On the moduli of continuity of $H^p$ functions with $0<p<1$},
Proc. Edinburgh Math. Soc. {\bf 35}(1992), 89--100.


\bibitem{mono}
M. Pavlovi\'c, \emph{Introduction to Function Spaces on the Disk,}
Mate\-mati\-\v cki institut u Beogradu, Posebna izdanja [Special Editions] 20, 2004.

\bibitem{decreasing} M. Pavlovi\'{c}, \emph{Analytic functions with decresing coefficients and Hardy spaces,} submitted.

\bibitem{novo} M. Pavlovi\'{c}, \emph{Definition and properties of the Libera operator on mixed norm spaces,} submitted.

 \bibitem{ru} L. A. Rubel, A. L. Shields, and B. A. Taylor, \emph{Mergelyan sets and the modulus of continuity
of analytic functions}, J. Approximation Theory {\bf 15} (1975), no. 1, 23--40.



\bibitem{SW} A.L. Shields and D.L. Williams, \emph{Bounded projections, duality, and multipliers in
spaces of analytic functions}, Trans. Amer. Math. Soc. {\bf 162} (1971), 287--302.

\bibitem{S} A.G. Siskasis, \emph{Composition semigroups and the Ces\`{a}ro operator on
$H^p$,} {J. London Math. Soc.} {\bf 36}(2)  (1987), 153--164.



\bibitem{S1} A. G. Siskakis, \emph{The Ces\`{a}ro operator is bounded on
$H^1$,} {Proc. Amer. Math. Soc. } {\bf 110} (1990), 461--462.


\bibitem{ta} P. M. Tamrazov, \emph{Contour and solid structural properties of holomorphic
functions of a complex variable} (Russian), Uspehi Mat. Nauk {\bf 28}
(1973), 131--161. English translation in Russian Math. Surveys {\bf 28}
(1973), 141--173.


\bibitem{xiao} Jie Xiao, \emph{Ces\`aro-type operators on Hardy, BMOA and Bloch spaces},
Arch.  Math. 68, No. 5 (1997),  398--406.


\end{thebibliography}

\end{document}